\documentclass[documentfleqn,12pt]{amsart}

\usepackage[margin=3cm]{geometry}

\usepackage{amsmath}
\usepackage{amssymb}
\usepackage{amsfonts}
\usepackage{enumerate}
\usepackage{graphicx}
\usepackage{color}
\usepackage{verbatim}

\usepackage{multicol}
\usepackage{mathrsfs}
\usepackage{mathtools}

\usepackage{subfig}
\usepackage{tikz}

\usetikzlibrary{arrows,shapes,positioning,decorations.markings}

\tikzset{->-/.style={decoration={
  markings,
  mark=at position .5 with {\arrow{>}}},postaction={decorate}}}

\usepackage{array}

\newtheorem{theorem}{Theorem}[section]
\newtheorem{definition}[theorem]{Definition}
\newtheorem{lemma}[theorem]{Lemma}

\newtheorem{conjecture}[theorem]{Conjecture}

\newtheorem{corollary}[theorem]{Corollary}

\newtheorem{proposition}[theorem]{Proposition}

\def\C{\mathbb{C}}
\def\R{\mathbb{R}}
\def\Z{\mathbb{Z}}

\begin{document}

\title[Generalized Catalan Numbers and the Enumeration of Planar Embeddings]{Generalized Catalan Numbers and the Enumeration of Planar Embeddings}

\author{Jonathan E. Beagley \and Paul Drube}

\address{Department of Mathematics and Computer Science\\ Valparaiso University, 1900 Chapel Drive, Valparaiso, IN 46383}

\email{jon.beagley@valpo.edu, paul.drube@valpo.edu}


\begin{abstract}
The Raney numbers $R_{p,r}(n)$ are a two-parameter generalization of the Catalan numbers that were introduced by Raney in his investigation of functional composition patterns \cite{Raney}.  We give a new combinatorial interpretation for all Raney numbers in terms of planar embeddings of certain collections of trees, a construction that recovers the usual interpretation of the $p$-Catalan numbers in terms of $p$-ary trees via the specialization $R_{p,1}(n) = \prescript{}{p}c_n$.  Our technique leads to several combinatorial identities involving the Raney numbers and ordered partitions.  We then give additional combinatorial interpretations of specific Raney numbers, including an identification of $R_{p^2,p}(n)$ with oriented trees whose vertices satisfy the ``source or sink property".  We close with comments applying these results to the enumeration of connected (non-elliptic) $A_2$ webs that lack an internal cycle.
\end{abstract}

\maketitle

\section{Introduction}
\label{sec: intro}

We investigate a two-parameter generalization of the Catalan numbers known as the \emph{Raney numbers}, as first studied by Raney \cite{Raney}.  These Raney numbers are defined as $R_{p,r}(n) = \displaystyle\frac{r}{np+r}\binom{np+r}{n}$ for all positive integers $n,p,r$, and specialize to both the usual Catalan numbers as $R_{2,1}(n) = c_n$ and to the $p$-Catalan numbers as $R_{p,1}(n) = \prescript{}{p}c_n$.  Raney numbers have previously seen applications to compositional patterns \cite{Raney} and probability theory \cite{Mlot}.  In this paper we give a new set of combinatorial interpretations for $R_{p,r}$ that directly generalize the well-known application of $p$-Catalan numbers to the enumeration of $p$-ary trees \cite{HP},\cite{stanley}.  Our work can also be interpreted as a generalization of the planted plane tree enumeration techniques developed by Harary, Prins, \& Tutte \cite{HPT} and Klarner \cite{Klarner}, and our results specialize to all of those tree enumeration results via specific choices of $p$ and $r$.

We begin in Section \ref{sec: primary results} with a careful description of our ``generalized $p$-ary trees", which are in one-to-one correspondence with planar embeddings of trees with specific vertex structures.  Section \ref{sec: primary results} goes on to provide two independent methods for counting these generalized trees: Proposition \ref{thm: coral enumeration partitions} is a ``tiered approach" that generalizes a more specialized  result of Klarner \cite{Klarner}, whereas Theorem \ref{thm: coral enumeration Raney numbers} is a modification of a construction by Hilton and Pedersen \cite{HP} that directly relates our enumerations to the Raney numbers.  Our two techniques are brought together by the combinatorial identity of Theorem \ref{thm: Raney numbers and ordered partitions}, which is summarized below:

\begin{theorem}
\label{thm: intro theorem 1}
Let $n$ be a positive integer.  Then for all positive integers $p,r$ we have:

\begin{center}
$\displaystyle{R_{p,r}(n) = \frac{r}{np+r}\binom{np+r}{n} = \sum_\lambda \binom{r}{\lambda_1} \binom{p \lambda_1}{\lambda_2} \binom{p \lambda_2}{\lambda_3} \hdots \binom{p \lambda_{j-1}}{\lambda_j}}$
\end{center}
\noindent Where $\lambda = (\lambda_1,\lambda_2,...\lambda_j)$ ranges over all ordered partitions of $n$.
\end{theorem}

In Section \ref{sec: interpretations} we apply our primary results to give combinatorial interpretations of Raney numbers for specific values of $p$ and $r$.  Our most innovative result in this section is our identification of $R_{p^2,p}(k)$ with edge-oriented trees whose $(p+1)$-valent vertices coherently obey the ``source or sink" property.  In the case of $p=2$ this gives an enumeration of connected, non-elliptic $A_2$ webs with no internal cycles: a significant subclass of the non-elliptic $A_2$ webs introduced by Kuperberg \cite{Kup} to graphically encode the representation theory of the quantum enveloping algebra $U_q(sl_3)$.  In Corollary 3.5 we eventually prove the following, which characterizes a certain subset of $Hom_{sl_3}(V^{\otimes 3(k+1)}, \C)$ for the three-dimensional irreducible $sl_3$-module $V$:

\begin{theorem}
\label{thm: intro theorem 2}
$R_{4,2}(k)$ equals the number of connected, non-elliptic $A_2$ webs that lack an internal face and have a constant boundary string with $3(k+1)$ pluses.
\end{theorem}

After giving another application of $R_{4,1}$ to a different class of non-elliptic $A_2$ webs, we close the paper with a series of conjectures that hope to generalize our results to $sl_n$ webs, which similarly that encode the representation theory of $U_q(sl_n)$.  In particular, we assert a correlation between the Raney number $R_{n+1,n-1}(k)$ and linearly-independent connected $sl_n$ webs that lack an internal cycle and have a boundary string corresponding to $n(k+1)$ total 1's.

\section{Raney Numbers \& the Enumeration of Planar Tree Embeddings}
\label{sec: primary results}

In this section we present the primary construction of this paper, which gives a geometric realization of the Raney numbers in terms of planar embeddings of certain types of trees.  We begin by introducing our graph theoretic terminology:

\begin{definition}
\label{def: star}
Let $p$ be a positive integer.  Then a \textbf{p-star} is a rooted tree with $p$ terminal edges lying above a single base vertex.
\end{definition}


In this section we will directly use $p$-stars as building blocks for larger graphs.  In Section \ref{sec: primary results} we will modify stars by allowing their edges to be directed, or by replacing the basic $p$-star with more complicated subgraphs that retain a single base vertex and $p$ terminal edges.  For a fixed $p$, $p$-stars are used to construct planar graphs that we refer to as coral diagrams:


\begin{definition}
\label{def: coral diagram}
Let $p,r$ be positive integers.  A \textbf{coral diagram of type $(p,r)$} is a rooted tree that is constructed from a $(r+1)$-valent base vertex via the repeated placement of $p$-stars atop terminal edges that are not the leftmost edge adjacent to the base vertex.
\end{definition}

\begin{figure}[h!]
\begin{tikzpicture}
	[scale=.5,auto=left,every node/.style={circle,fill=black,inner sep=1.7pt}]
	\node (n1) at (2,0) {};
	\node (n2) at (0,4.5) {};
	\node (n3) at (2,1.5) {};
	\node (n4) at (4,1.5) {};
	\draw[thick,bend left=25] (n1) to (n2);
	\draw[thick] (n1) to (n3);
	\draw[thick] (n1) to (n4);
\end{tikzpicture}
\hspace{.15in}
\scalebox{3}{\raisebox{12pt}{$\Rightarrow$}}
\hspace{.15in}
\begin{tikzpicture}
	[scale=.5,auto=left,every node/.style={circle,fill=black,inner sep=1.7pt}]
	\node (n1) at (2,0) {};
	\node (n2) at (0,4.5) {};
	\node (n3) at (2,1.5) {};
	\node (n4) at (4,1.5) {};
	\node (n5) at (1.3,3.0) {};
	\node (n6) at (2.7,3.0) {};
	\node (n7) at (3.3,3.0) {};
	\node (n8) at (4.7,3.0) {};
	\draw[thick,bend left=25] (n1) to (n2);
	\draw[thick] (n1) to (n3);
	\draw[thick] (n1) to (n4);
	\draw[thick] (n3) to (n5);
	\draw[thick] (n3) to (n6);
	\draw[thick] (n4) to (n7);
	\draw[thick] (n4) to (n8);
\end{tikzpicture}
\hspace{.15in}
\scalebox{3}{\raisebox{12pt}{$\Rightarrow$}}
\hspace{.15in}
\begin{tikzpicture}
	[scale=.5,auto=left,every node/.style={circle,fill=black,inner sep=1.7pt}]
	\node (n1) at (2,0) {};
	\node (n2) at (0,4.5) {};
	\node (n3) at (2,1.5) {};
	\node (n4) at (4,1.5) {};
	\node (n5) at (1.3,3.0) {};
	\node (n6) at (2.7,3.0) {};
	\node (n7) at (3.3,3.0) {};
	\node (n8) at (4.7,3.0) {};
	\node (n9) at (2.0,4.5) {};
	\node (n10) at (3.4,4.5) {};
	\draw[thick,bend left=25] (n1) to (n2);
	\draw[thick] (n1) to (n3);
	\draw[thick] (n1) to (n4);
	\draw[thick] (n3) to (n5);
	\draw[thick] (n3) to (n6);
	\draw[thick] (n4) to (n7);
	\draw[thick] (n4) to (n8);
	\draw[thick] (n6) to (n9);
	\draw[thick] (n6) to (n10);
\end{tikzpicture}
\caption{Construction of a (2,2)-coral diagram with three 2-stars}
\label{fig: coral diagram}
\end{figure}
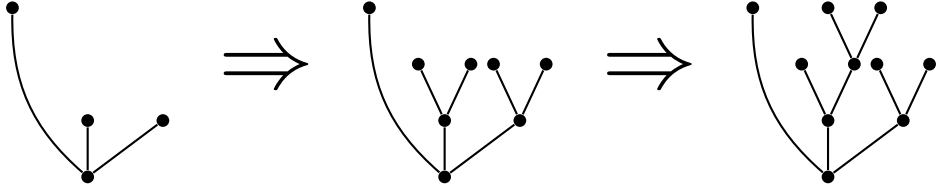

We refer to the vertices that serve as base points for our $p$-stars as the \emph{attachment sites} of our coral diagram.  In our upcoming constructions it will be useful to think of coral diagrams as a collection of $k+1$ trees ($k$ $p$-stars and a single ``base" $(r+1)$-star) where we have identified the vertices corresponding to the $k$ attachment sites.


The condition from Definition \ref{def: coral diagram} that we cannot add stars to the leftmost initial edge in a coral diagram is absolutely essential for our combinatorial interpretations.  Our primary concern is planar embeddings of graphs, with equivalence given by homeomorphisms that fix a linear ordering of the terminal vertices.  When associating planar embeddings with rooted trees, the primary difficulty is consistently dealing with the fact that a single embedding may be rooted at multiple distinct vertices.  Not attaching stars to the leftmost edge of our coral diagram gives us a consistent way of selecting a base vertex for our planar embeddings and ensures that homeomorphic trees cannot correspond to distinct coral diagrams: each equivalence class of homeomorphic planar embeddings contains a unique member where there is a single edge connecting the leftmost boundary point with a ``bottom" vertex.  Although somewhat awkward from a diagrammatic perspective, this way of selecting root vertices will allow for a particularly elegant correspondence between coral diagrams and Raney numbers in Theorem \ref{thm: coral enumeration Raney numbers}.

In Figure \ref{fig: tree rotation} we show the (2,2)-coral diagram from Figure \ref{fig: coral diagram} and a homeomorphic embedding of the same graph \footnote{Assume in all of our figures that terminal edges have been extended to a fixed boundary line at the top of the diagram.  In all constructions we assume a boundary that is homeomorphic to $\R$}.  Notice that the embedding on the right does not represent a valid coral diagram of any type.

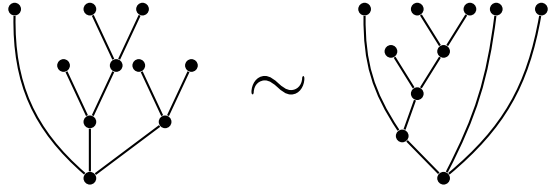
\begin{figure}[h!]
\begin{tikzpicture}
	[scale=.5,auto=left,every node/.style={circle,fill=black,inner sep=1.7pt}]
	\node (n1) at (2,0) {};
	\node (n2) at (0,4.5) {};
	\node (n3) at (2,1.5) {};
	\node (n4) at (4,1.5) {};
	\node (n5) at (1.3,3.0) {};
	\node (n6) at (2.7,3.0) {};
	\node (n7) at (3.3,3.0) {};
	\node (n8) at (4.7,3.0) {};
	\node (n9) at (2.0,4.5) {};
	\node (n10) at (3.4,4.5) {};
	\draw[thick,bend left=25] (n1) to (n2);
	\draw[thick] (n1) to (n3);
	\draw[thick] (n1) to (n4);
	\draw[thick] (n3) to (n5);
	\draw[thick] (n3) to (n6);
	\draw[thick] (n4) to (n7);
	\draw[thick] (n4) to (n8);
	\draw[thick] (n6) to (n9);
	\draw[thick] (n6) to (n10);
\end{tikzpicture}
\hspace{.15in}
\scalebox{2.5}{\raisebox{12pt}{$\sim$}}
\hspace{.15in}
\begin{tikzpicture}
	[scale=.5,auto=left,every node/.style={circle,fill=black,inner sep=1.7pt}]
	\node (n1) at (2.1,0.0) {};
	\node (n2) at (1.0,1.125) {};
	\node (n3) at (0.0,4.5) {};
	\node (n4) at (0.7,3.375) {};
	\node (n5) at (1.4,2.25) {};
	\node (n6) at (1.4,4.5) {};
	\node (n7) at (2.1,3.375) {};
	\node (n8) at (2.8,4.5) {};
	\node (n9) at (3.5,4.5) {};
	\node (n10) at (4.7,4.5) {};
	\draw[thick] (n1) to (n2);
	\draw[thick, bend left=17] (n2) to (n3);
	\draw[thick] (n2) to (n5);
	\draw[thick] (n5) to (n4);
	\draw[thick] (n5) to (n7);
	\draw[thick] (n7) to (n6);
	\draw[thick] (n7) to (n8);
	\draw[thick, bend right=10] (n1) to (n9);
	\draw[thick, bend right=20] (n1) to (n10);
\end{tikzpicture}
\caption{A (2,2)-coral diagram and an equivalent planar embedding}
\label{fig: tree rotation}
\end{figure}

We are now ready for our primary results enumerating coral diagrams.  Henceforth denote the number of distinct $(p,r)$-coral diagrams with exactly $k$ total $p$-stars added to the base star by $T_{p,r}(k)$.  The proposition below is the more direct of our two methodologies, and will be used in Theorem \ref{thm: Raney numbers and ordered partitions} to introduce an entirely new combinatorial identity involving the Raney numbers.  Notice that Proposition \ref{thm: coral enumeration partitions} closely resembles (20) from \cite{Klarner} after specializing to $r=1$.

\begin{proposition}
\label{thm: coral enumeration partitions}
Let $p,r$ be positive integers.  Then the number of coral diagrams of type $(p,r)$ with precisely $k$ p-stars is:

\begin{center}
$\displaystyle{T_{p,r}(k) = \sum_\lambda \binom{r}{\lambda_1} \binom{p \lambda_1}{\lambda_2} \binom{p \lambda_2}{\lambda_3} \hdots \binom{p \lambda_{j-1}}{\lambda_j}}$
\end{center}

\noindent Where $\lambda = (\lambda_1,\lambda_2,...\lambda_j)$ varies over all ordered partitions of $k$ (of any length $j \geq 1$).
\end{proposition}
\begin{proof}
As in Figure \ref{fig: coral diagram}, we will construct our coral diagram by attaching our $k$ stars one ``tier" at a time, beginning with the base tree and working upward.  We assume that all stars are attached as low as possible, so that if a vertex is not used as an attachment point for a given tier that vertex cannot serve as an attachment point for later tiers.  Hence the only valid attachment points at each tier are the terminal vertices of stars added in the previous tier.

So let $\lambda = (\lambda_1, \lambda_2, \hdots, \lambda_j)$ be an ordered partition of $k$, and consider the number of $(p,r)$-coral diagrams with $\lambda_1$ stars attached in the first tier, $\lambda_2$ stars attached in the second tier, etc.  For the first tier there are $r$ available attachment sites, corresponding to the $r$ terminal edges adjacent to the base that are not the leftmost edge.  Hence there are $\binom{r}{\lambda_1}$ distinct ways to attach stars of this tier.  For the $j$\textsuperscript{th} tier ($j > 1$) there are $p \lambda_{j-1}$ valid attachment sites, corresponding to the top vertices of the $\lambda_{j-1}$ $p$-stars from the previous tier.  Thus there are $\binom{p \lambda_{j-1}}{\lambda_j}$ distinct ways to attach stars of this tier.

As the resulting diagrams are rooted, and since we are not allowing star attachment to the leftmost initial edge, all trees produced in this manner are non-equivalent embeddings.  This leaves $\binom{r}{\lambda_1} \binom{p \lambda_1}{\lambda_2} \binom{p \lambda_2}{\lambda_3} \hdots \binom{p \lambda_{j-1}}{\lambda_j}$ distinct trees corresponding to our partition $\lambda$.  As coral diagrams produced from distinct partitions are clearly non-equivalent, this proves the result.
\end{proof}

Our second construction directly relates the number of $(p,r)$-coral diagrams to the Raney numbers $R_{p,r}(k)$.  We preface our result with a characterization of the Raney numbers that is proven in by Hilton and Pedersen in \cite{HP}. \footnote{Our Raney number $R_{p,r}(n)$ corresponds to $_pd_{qk} = d_{qk}$ in \cite{HP} via $q = p-r, k=n+1$}

\begin{lemma}\cite[Theorem 2.6]{HP}
\label{thm: HiltonPederson lemma}
Let $_pc_k = \frac{1}{k} \binom{pk}{k-1}$ be the $k$\textsuperscript{th} $p$-Catalan number, and let $R_{p,r}(k)$ denote the Raney number.  Then:

\begin{center}
$\displaystyle{R_{p,r}(k) = \sum_{i_1 + \hdots + i_r = k} \ _pc_{i_1} \ _pc_{i_2} \ \hdots \ _pc_{i_r}}$
\end{center}
\end{lemma}

\begin{theorem}
\label{thm: coral enumeration Raney numbers}
Let $p,r$ be positive integers.  Then the number of $(p,r)$-coral diagrams with precisely $k$ p-stars equals the k\textsuperscript{th} evaluation of the Raney number $R_{p,r}$:

\begin{center}
$\displaystyle{T_{p,r}(k) = R_{p,r}(k) = \binom{pk+r-1}{k-1} \frac{r}{k}}$
\end{center}
\end{theorem}
\begin{proof}
It is well established (see \cite{HP}) that $_pc_j$ equals the number of connected trees constructed from $j$ total $p$-stars.  With this interpretation, Lemma \ref{thm: HiltonPederson lemma} states that the Raney number $R_{p,r}(k)$ counts the number of distinct ordered $r$-tuples of trees constructed from p-stars such that a total of k total p-stars are utilized across the entire tuple.

Now consider a $(p,r)$-coral diagram.  We may subdivide the coral diagram into $r+1$ subgraphs (some of which may be empty): one corresponding to the base $r$-star and one corresponding to everything added atop each of the $r$ attachment sites of the base $(r+1)$-star.  Any coral diagram may then be described by an $r$-tuple of trees constructed from $p$-stars.  As the leftmost edge of the base star still isn't a valid attachment site, it isn't possible to achieve equivalent coral diagrams from distinct $r$-tuples of trees.  If we fix the total number of $p$-stars to be used at $k \geq 0$, it follows that $T_{p,r} = \displaystyle{\sum_{i_1 + \hdots + i_r = k} \ _pc_{i_1} \ _pc_{i_2} \ \hdots \ _pc_{i_r}}$.  Lemma \ref{thm: HiltonPederson lemma} then gives the desired result.
\end{proof}

As a quick corollary of Proposition \ref{thm: coral enumeration partitions} and Theorem \ref{thm: coral enumeration Raney numbers}, we have the following combinatorial identity that relates the Raney numbers to ordered partitions

\begin{theorem}
\label{thm: Raney numbers and ordered partitions}
Let $p,r$ be positive integers, and consider the Raney number $R_{p,r}$.  Then:

\begin{center}
$\displaystyle{R_{p,r}(k) = \sum_\lambda \binom{r}{\lambda_1} \binom{p \lambda_1}{\lambda_2} \binom{p \lambda_2}{\lambda_3} \hdots \binom{p \lambda_{j-1}}{\lambda_j}}$
\end{center}
\noindent Where $\lambda = (\lambda_1,\lambda_2,...\lambda_j)$ varies over all ordered partitions of $k$ (of any length $j \geq 1$).
\end{theorem}

Notice that this is a distinct reduction of $R_{p,r}(k)$ into a summations over ordered partitions of $k$ than the one presented in Lemma \ref{thm: HiltonPederson lemma}, even though both summations involve terms of the form $\binom{p\lambda_j}{m}$.  The decomposition of Lemma \ref{thm: HiltonPederson lemma} ranges over weak partitions and follows from ``horizontally" dividing our coral diagram into a tuple of attached trees, whereas the summation of Theorem \ref{thm: Raney numbers and ordered partitions} ranges over (strong) partitions and follows from ``vertically" dividing our coral diagram into tiers. 

\section{Combinatorial Interpretations of $R_{p,r}$}
\label{sec: interpretations}

The remainder of this paper is devoted to combinatorial interpretations of the Raney numbers $R_{p,r}$ for specific choices of $p,r$.  Our primary tools are Theorem \ref{thm: coral enumeration Raney numbers} and the coral diagram framework that it suggests.  To begin with the least derived (and least informative) result we have:

\begin{proposition}
\label{thm: raney interpretation general p r}
Let $p,r$ be positive integers, then $R_{p,r}(k)$ equals the number of distinct planar embeddings of trees with $k+1$ internal vertices such that all internal vertices are $(p+1)$-valent except for the vertex incident upon the leftmost terminal edge, which is $(r+1)$-valent.
\end{proposition}
\begin{proof}
Follows directly from Proposition \ref{thm: coral enumeration Raney numbers} and the definition of coral diagram.
\end{proof}

Obviously, the situation becomes far more interesting if $p=r$, with $R_{p,p}(k)$ enumerating distinct planar embeddings of wholly $(p+1)$-valent trees.  This specialization also yields a new proof of the following Raney number identity:

\begin{proposition}
\label{thm: raney interpretation p=r}
Let $p$ be a positive integer.  Then $R_{p,p}(k) = R_{p,1}(k+1)$, with both quantities equaling the number of distinct planar embeddings of $(p+1)$-valent trees with $k+1$ internal vertices.
\end{proposition}
\begin{proof}
	Consider a $(p,p)$-coral diagram with $k$ total $p$-stars.  We divide the leftmost edge emanating from the base vertex by adding an additional 2-valent vertex, and then isotope so that this new vertex lies at the base.  This is now a $(p,1)$-coral diagram with $k+1$ total $p$-stars.  Since there is only one attachment site for the first $p$-star, we are able to get every coral diagram of type $(p,1)$ with $k+1$ stars in this manner.  Also, it is clear that distinct $(p,p)$-coral diagrams with $k$ $p$-stars are transformed into distinct $(p,1)$-coral diagrams with $k+1$ $p$-stars.
\end{proof}

Also of interest is the situation where $p=1$, where coral diagrams provide a new ``geometric" proof of the result that $R_{1,r}(k)$ is related to ordered weak partitions of $k$:

\begin{proposition}
\label{thm: raney interpretation p=1}
Let $r$ be a positive integer.  Then $R_{1,r}(k)$ equals the number of ordered weak partitions of the positive integer $k$ into $r$ pieces.
\end{proposition}
\begin{proof}
$(1,r)$-coral diagrams take the form shown in Figure \ref{fig: (1,r) coral diagram}.  Adding $k$ total 1-stars then amounts to choosing a partition of $k$ into $r$ pieces with $\lambda_j \geq 0$ 1-stars each.  This partition is ordered because of the unused edge at the left.
\end{proof}

\begin{figure}[h!]
\centering
\begin{tikzpicture}
	[scale=.45,auto=left,every node/.style={circle,fill=black,inner sep=1.7pt}]
	\node (base) at (2.75,0) {};
	\node (a) at (0,6.0) {};
	\node (b1) at (2,1.5) {};
	\node (b2) at (2,2.75) {};
	\node[fill=white] (b3) at (2,3.65) {\scalebox{.85}{$\vdots$}};
	\node (b4) at (2,4.25) {};
	\node (b5) at (2,5.5) {};
	\node (c1) at (3.5,1.5) {};
	\node (c2) at (3.5,2.75) {};
	\node[fill=white] (c3) at (3.5,3.65) {\scalebox{.85}{$\vdots$}};
	\node (c4) at (3.5,4.25) {};
	\node (c5) at (3.5,5.5) {};
	\node (d1) at (5,1.5) {};
	\node (d2) at (5,2.75) {};
	\node[fill=white] (d3) at (5,3.65) {\scalebox{.85}{$\vdots$}};
	\node (d4) at (5,4.25) {};
	\node (d5) at (5,5.5) {};
	\draw[thick,bend left=30] (base) to (a);
	\draw[thick,bend left=5] (base) to (b1);
	\draw[thick] (b1) to (b2);
	\draw[thick] (b4) to (b5);
	\draw[thick,bend right=10] (base) to (c1);
	\draw[thick] (c1) to (c2);
	\draw[thick] (c4) to (c5);
	\draw[thick,bend right=25] (base) to (d1);
	\draw[thick] (d1) to (d2);
	\draw[thick] (d4) to (d5);
\end{tikzpicture}
\caption{A (1,3)-coral diagram, giving an ordered partition in 3 pieces}
\label{fig: (1,r) coral diagram}
\end{figure}
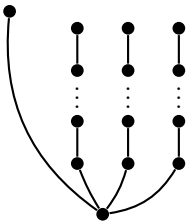

Further interpretations of the Raney numbers can be made when one orients the edges of the planar embeddings.  Here we only consider orientations that are coherent in the sense that every vertex in either a source or a sink.

\begin{proposition}
\label{thm: raney interpretation p=r2}
Let $p$ be a positive integer, then $R_{p^2,p}(k)$ equals the number of distinct planar embeddings of $(p+1)$-valent trees, coherently oriented according to the source-sink property,  with $k(p^2-1) + (p+1)$ 1-valent boundary vertices that are all sinks.
\end{proposition}
\begin{proof}
We look to establish a bijection between planar embeddings that satisfy the hypothesis of the proposition and $(p^2,p)$-coral diagrams whose ``stars" take the modified form shown in Figure \ref{fig: oriented p2 stars}.  We induct on the number of internal vertices in the tree:

For the base step, notice that a $(p+1)$-valent tree with one internal vertex and $(p+1)$ external sinks is merely the base of a $(p^2,p)$-coral diagram where the base vertex is a source.  For the inductive step, notice that adding $p$ edges to an external sink produces $p$ new external vertices that are all sources.  If the resulting tree is to satisfy the hypotheses of the proposition, this necessitates the addition of $p$ additional edges to each of these new vertices.  This yields $p^2$ new external sinks, and shows that any qualifying tree must be built up via the attachment of entire $p^2$-stars of the type shown in Figure \ref{fig: oriented p2 stars}.  It follows that every planar embedding created in this way is homeomorphic to a unique $(p^2,p)$-coral diagram.
\end{proof}

Figure \ref{fig: oriented p2 stars} shows our modified $p^2$-stars for source-sink directed trees in the cases of $p=2$ and $p=3$.  Notice that the top of those stars exhibit a constant upward orientation, and that the base of those stars can only be attached to an upward oriented edge.

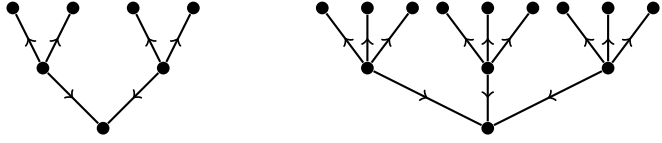
\begin{figure}[h!]
\centering
\begin{tikzpicture}
	[scale=.8,auto=left,every node/.style={circle,fill=black,inner sep=1.7pt}]
	\node (n1) at (0,0) {};
	\node (n2) at (-1,1) {};
	\node (n3) at (1,1) {};
	\node (n4) at (-1.5,2) {};
	\node (n5) at (-.5,2) {};
	\node (n6) at (.5,2) {};
	\node (n7) at (1.5,2) {};
	\draw[->-, thick] (n2) to (n1);
	\draw[->-, thick] (n3) to (n1);
	\draw[->-, thick] (n2) to (n4);
	\draw[->-, thick] (n2) to (n5);
	\draw[->-, thick] (n3) to (n6);
	\draw[->-, thick] (n3) to (n7);
\end{tikzpicture}
\hspace{.5in}
\begin{tikzpicture}
	[scale=.8,auto=left,every node/.style={circle,fill=black,inner sep=1.7pt}]
	\node (n1) at (0,0) {};
	\node (n2) at (-2,1) {};
	\node (n3) at (0,1) {};
	\node (n4) at (2,1) {};
	\node (n5) at (-2.75,2) {};
	\node (n6) at (-2,2) {};
	\node (n7) at (-1.25,2) {};
	\node (n8) at (-.75,2) {};
	\node (n9) at (0,2) {};
	\node (n10) at (.75,2) {};
	\node (n11) at (1.25,2) {};
	\node (n12) at (2,2) {};
	\node (n13) at (2.75,2) {};
	\draw[->-, thick] (n2) to (n1);
	\draw[->-, thick] (n3) to (n1);
	\draw[->-, thick] (n4) to (n1);
	\draw[->-, thick] (n2) to (n5);
	\draw[->-, thick] (n2) to (n6);
	\draw[->-, thick] (n2) to (n7);
	\draw[->-, thick] (n3) to (n8);
	\draw[->-, thick] (n3) to (n9);
	\draw[->-, thick] (n3) to (n10);
	\draw[->-, thick] (n4) to (n11);
	\draw[->-, thick] (n4) to (n12);
	\draw[->-, thick] (n4) to (n13);
\end{tikzpicture}
\caption{Oriented ``stars" for $(2^2,2)$- and $(3^2,3)$-coral diagrams}
\label{fig: oriented p2 stars}
\end{figure}

A quick inductive argument shows that the only $(p+1)$-valent trees with the properties required by Proposition \ref{thm: raney interpretation p=r2} have $k(p^2-1)+(p+1)$ boundary points.  Thus, ranging over $R_{p^2,p}(k)$ for all $k \geq 0$ accounts for all source-sink oriented (p+1)-valent trees with constant boundary vertex orientation.

An equivalent interpretation to the one in Proposition \ref{thm: raney interpretation p=r2} is that $R_{p^2,p}(k)$ counts the number of distinct planar embeddings of $(p+1)$-valent trees with $k(p^2-1) + (p+1)$ 1-valent boundary points such that, for any fixed vertex, any path from that vertex to the boundary passes through an equivalent number of edges modulo two.

The primary reason we present the specific interpretation of Proposition \ref{thm: raney interpretation p=r2} is that the $(p,r)=(4,2)$ case gives planar embeddings that qualify as (non-elliptic) $A_2$ webs (referred to by some as simply $sl_3$ webs).  $A_2$ webs constitute the morphisms in the braided monoidal category introduced by Kuperberg \cite{Kup} to diagrammatically present the representation theory of the quantum enveloping algebra $U_q(sl_3)$.  Objects in this category are finite words in the alphabet $\lbrace +,-\rbrace$, corresponding to the two (dual) irreducible three-dimensional $sl_3$-modules $V^+$ and $V^-$.  These words are encoded via the orientation of the boundary vertices, so that all webs can be represented as elements of $Hom(\vec{s},\emptyset)$ for boundary word $\vec{s}$.  Non-elliptic webs are those webs that lack internal square and bigons.  Non-elliptic webs form a linear basis for all $A_2$ webs over $\Z[q,q^{-1}]$.  For a constant boundary string of $3k$ total +'s, the total number of non-elliptic webs is known to be in bijection with standard Young tableaux of size $3 \times k$ (\cite{PPS},\cite{Tym}).  Our $(4,2)$-coral diagrams with $k$ total $p$-stars then form a subset of non-elliptic webs with a constant boundary string of $k(4-1) + (2+1) = 3(k+1)$ total +'s, and are in bijection with an interesting subset of standard Young tableaux of size $3 \times (k+1)$.

In terms of $A_2$ web terminology, Proposition \ref{thm: raney interpretation p=r2} can be restated in the specific case of $p=2$ as follows:

\begin{corollary}
\label{thm: sl3 webs, constant string}
$R_{4,2}(k)$ equals the number of connected (non-elliptic) $A_2$ webs that lack an internal face and have a constant boundary string with $3(k+1)$ pluses.
\end{corollary}

In an upcoming paper \cite{BD}, the authors apply Corollary \ref{thm: sl3 webs, constant string} and other combinatorial results to enumerate webs with distinct geometric structures.  In Figure \ref{fig: sl3 webs} we show how Corollary \ref{thm: sl3 webs, constant string} applies to $A_2$ webs with boundary word $(++++++)$.  Here we have 5 total non-elliptic webs, $R_{4,2}(1)=2$ of which are connected trees.

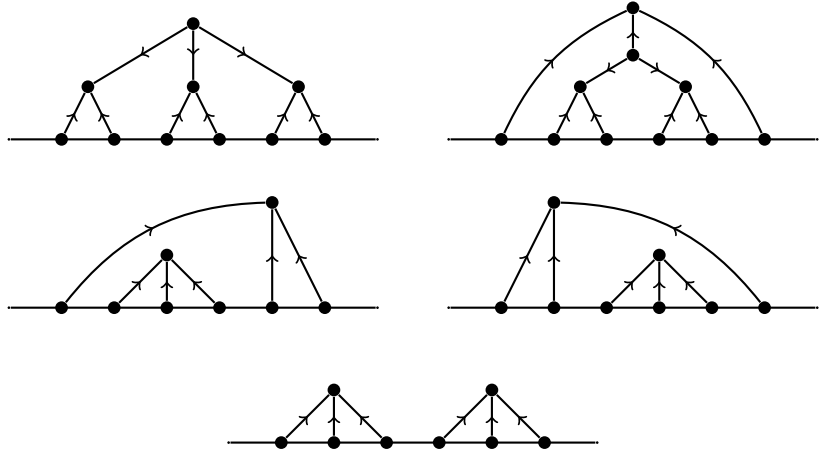
\begin{figure}[h!]

\begin{tikzpicture}
	[scale=.7,auto=left,every node/.style={circle,fill=black,inner sep=1.7pt}]
	\node[inner sep=0pt] (left) at (0,0) {};
	\node (+1) at (1,0) {};
	\node (+2) at (2,0) {};
	\node (+3) at (3,0) {};
	\node (+4) at (4,0) {};
	\node (+5) at (5,0) {};
	\node (+6) at (6,0) {};
	\node[inner sep=0pt] (right) at (7,0) {};
	\draw[thick] (left) to (right);

	\node (a) at (1.5,1) {};
	\node (b) at (3.5,1) {};
	\node (c) at (5.5,1) {};
	\node (d) at (3.5,2.2) {};
	\draw[->-, thick] (+1) to (a);
	\draw[->-, thick] (+2) to (a);
	\draw[->-, thick] (+3) to (b);
	\draw[->-, thick] (+4) to (b);
	\draw[->-, thick] (+5) to (c);
	\draw[->-, thick] (+6) to (c);	
	\draw[->-, thick] (d) to (a);	
	\draw[->-, thick] (d) to (b);	
	\draw[->-, thick] (d) to (c);		
\end{tikzpicture}
\hspace{.25in}
\begin{tikzpicture}
	[scale=.7,auto=left,every node/.style={circle,fill=black,inner sep=1.7pt}]
	\node[inner sep=0pt] (left) at (0,0) {};
	\node (+1) at (1,0) {};
	\node (+2) at (2,0) {};
	\node (+3) at (3,0) {};
	\node (+4) at (4,0) {};
	\node (+5) at (5,0) {};
	\node (+6) at (6,0) {};
	\node[inner sep=0pt] (right) at (7,0) {};
	\draw[thick] (left) to (right);

	\node (a) at (2.5,1) {};
	\node (b) at (4.5,1) {};
	\node (c) at (3.5,2.5) {};
	\node (d) at (3.5,1.6) {};
	\draw[->-, thick,bend left=20] (+1) to (c);
	\draw[->-, thick] (+2) to (a);
	\draw[->-, thick] (+3) to (a);
	\draw[->-, thick] (+4) to (b);
	\draw[->-, thick] (+5) to (b);
	\draw[->-, thick,bend right=20] (+6) to (c);
	\draw[->-, thick] (d) to (a);
	\draw[->-, thick] (d) to (b);
	\draw[->-, thick] (d) to (c);
\end{tikzpicture}

\vspace{.25in}

\begin{tikzpicture}
	[scale=.7,auto=left,every node/.style={circle,fill=black,inner sep=1.7pt}]
	\node[inner sep=0pt] (left) at (0,0) {};
	\node (+1) at (1,0) {};
	\node (+2) at (2,0) {};
	\node (+3) at (3,0) {};
	\node (+4) at (4,0) {};
	\node (+5) at (5,0) {};
	\node (+6) at (6,0) {};
	\node[inner sep=0pt] (right) at (7,0) {};
	\draw[thick] (left) to (right);

	\node (a) at (3,1) {};
	\node (b) at (5,2) {};
	\draw[->-, thick,bend left=25] (+1) to (b);
	\draw[->-, thick] (+2) to (a);
	\draw[->-, thick] (+3) to (a);
	\draw[->-, thick] (+4) to (a);
	\draw[->-, thick] (+5) to (b);
	\draw[->-, thick] (+6) to (b);	
\end{tikzpicture}
\hspace{.25in}
\begin{tikzpicture}
	[scale=.7,auto=left,every node/.style={circle,fill=black,inner sep=1.7pt}]
	\node[inner sep=0pt] (left) at (0,0) {};
	\node (+1) at (1,0) {};
	\node (+2) at (2,0) {};
	\node (+3) at (3,0) {};
	\node (+4) at (4,0) {};
	\node (+5) at (5,0) {};
	\node (+6) at (6,0) {};
	\node[inner sep=0pt] (right) at (7,0) {};
	\draw[thick] (left) to (right);

	\node (a) at (4,1) {};
	\node (b) at (2,2) {};
	\draw[->-, thick] (+1) to (b);
	\draw[->-, thick] (+2) to (b);
	\draw[->-, thick] (+3) to (a);
	\draw[->-, thick] (+4) to (a);
	\draw[->-, thick] (+5) to (a);
	\draw[->-, thick,bend right=25] (+6) to (b);	
\end{tikzpicture}

\vspace{.35in}

\begin{tikzpicture}
	[scale=.7,auto=left,every node/.style={circle,fill=black,inner sep=1.7pt}]
	\node[inner sep=0pt] (left) at (0,0) {};
	\node (+1) at (1,0) {};
	\node (+2) at (2,0) {};
	\node (+3) at (3,0) {};
	\node (+4) at (4,0) {};
	\node (+5) at (5,0) {};
	\node (+6) at (6,0) {};
	\node[inner sep=0pt] (right) at (7,0) {};
	\draw[thick] (left) to (right);

	\node (a) at (2,1) {};
	\node (b) at (5,1) {};
	\draw[->-, thick] (+1) to (a);
	\draw[->-, thick] (+2) to (a);
	\draw[->-, thick] (+3) to (a);
	\draw[->-, thick] (+4) to (b);
	\draw[->-, thick] (+5) to (b);
	\draw[->-, thick] (+6) to (b);	
\end{tikzpicture}

\caption{The 5 non-elliptic $sl_3$ webs with boundary $(++++++)$}
\label{fig: sl3 webs}
\end{figure}

Raney numbers can also be applied to enumerate (non-elliptic) $A_2$ webs with boundary word of the form $(-+++\hdots +)$.  The construction here is similar to that in Proposition \ref{thm: raney interpretation p=r2}, with the same form of ``modified" 4-stars.  The only difference is that we now give the leftmost unused edge the opposite orientation and directly wrap it around to form a single attachment site for our 4-stars (as opposed to adding a single trivalent vertex in the base that flips the orientation of attachment sites).  This directly proves the following:

\begin{proposition}
\label{thm: sl3 webs, non-constant string}
$R_{4,1}(k)$ equals the number of connected (non-elliptic) $A_2$ webs that lack an internal face and have a constant boundary string with one minus followed by $3k+1$ pluses.
\end{proposition}

An application of Proposition \ref{thm: sl3 webs, non-constant string} in the case of $k=1$ is shown in Figure \ref{fig: sl3 webs non-constant}.  Here we have 3 total non-elliptic webs with boundary word $(-++++)$, $R_{4,1}(1)=1$ of which is a connected tree.  Notice that, since $R_{4,1}(k)$ is simply the k\textsuperscript{th} entry in the 4-Catalan sequence, this result also gives a new interpretation of the 4-Catalan numbers.

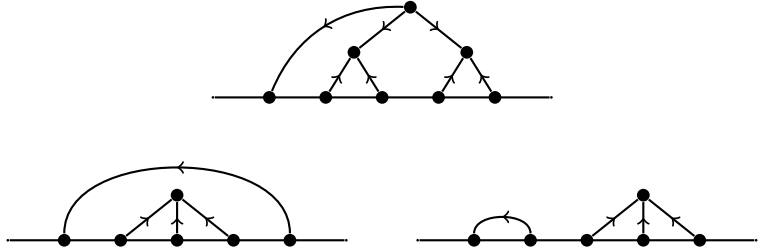
\begin{figure}[h!]

\begin{tikzpicture}
	[scale=.75,auto=left,every node/.style={circle,fill=black,inner sep=1.7pt}]
	\node[inner sep=0pt] (left) at (0,0) {};
	\node (-1) at (1,0) {};
	\node (+2) at (2,0) {};
	\node (+3) at (3,0) {};
	\node (+4) at (4,0) {};
	\node (+5) at (5,0) {};
	\node[inner sep=0pt] (right) at (6,0) {};
	\draw[thick] (left) to (right);

	\node (n1) at (2.5,.8) {};
	\node (n2) at (4.5,.8) {};
	\node (n3) at (3.5,1.6) {};
	\draw[->-, thick] (+2) to (n1);
	\draw[->-, thick] (+3) to (n1);
	\draw[->-, thick] (+4) to (n2);
	\draw[->-, thick] (+5) to (n2);
	\draw[->-, thick] (n3) to (n1);
	\draw[->-, thick] (n3) to (n2);
	\draw[->-, thick, bend right=35] (n3) to (-1);

\end{tikzpicture}

\vspace{.2in}

\begin{tikzpicture}
	[scale=.75,auto=left,every node/.style={circle,fill=black,inner sep=1.7pt}]
	\node[inner sep=0pt] (left) at (0,0) {};
	\node (-1) at (1,0) {};
	\node (+2) at (2,0) {};
	\node (+3) at (3,0) {};
	\node (+4) at (4,0) {};
	\node (+5) at (5,0) {};
	\node[inner sep=0pt] (right) at (6,0) {};
	\draw[thick] (left) to (right);

	\node (n1) at (3,.8) {};
	\draw[->-, thick] (+2) to (n1);
	\draw[->-, thick] (+3) to (n1);
	\draw[->-, thick] (+4) to (n1);
	\draw[->-, thick, bend right= 90] (+5) to (-1) {};

\end{tikzpicture}
\hspace{.25in}
\begin{tikzpicture}
	[scale=.75,auto=left,every node/.style={circle,fill=black,inner sep=1.7pt}]
	\node[inner sep=0pt] (left) at (0,0) {};
	\node (-1) at (1,0) {};
	\node (+2) at (2,0) {};
	\node (+3) at (3,0) {};
	\node (+4) at (4,0) {};
	\node (+5) at (5,0) {};
	\node[inner sep=0pt] (right) at (6,0) {};
	\draw[thick] (left) to (right);

	\node (n1) at (4,.8) {};
	\draw[->-, thick] (+3) to (n1);
	\draw[->-, thick] (+4) to (n1);
	\draw[->-, thick] (+5) to (n1);	
	\draw[->-, thick, bend right= 90] (+2) to (-1) {};
\end{tikzpicture}

\caption{The 3 non-elliptic $sl_3$ webs with boundary $(-++++)$}
\label{fig: sl3 webs non-constant}
\end{figure}


We close this paper by commenting upon possible interpretations of $R_{p,r}$ for other choices of $p,r$.  We focus on how our results about $A_2$ webs in Corollary \ref{thm: sl3 webs, constant string} and Proposition \ref{thm: sl3 webs, non-constant string} may be generalized to enumerate connected trees in the wider class of $sl_n$ webs.

Much as $A_2$ webs are used to diagrammatically present the representation theory of $U_q(sl_3)$, $sl_n$ webs are used to diagrammatically present the representation theory of $U_q(sl_n)$.  In addition to carrying an orientation, edges in these $sl_n$ webs are now labelled by one of the $(n-1)$ fundamental representations, while reversal of orientation corresponds to taking the dual of the given representation (in $sl_3$ webs there are 2 fundamental representations that are duals of one another, so the additional edge labellings are dropped because the orientations carry all necessary information).  The vertices of these webs must obey a more complicated set of conditions that depend on both orientation and edge-labelling, giving a far more complicated theory.  See \cite{Mor} for a detailed introduction to this topic, \cite{Fon} and \cite{West} for constructive algorithms producing $sl_n$ web bases, and the more recent \cite{CKM} for a generating set of relations for $sl_n$ webs.

The following conjecture is a direct generalization of Corollary \ref{thm: sl3 webs, constant string} to the enumeration of $sl_n$ ``tree webs" with a constant boundary string.

\begin{conjecture}
\label{thm: sln webs, constant}
For any $n \geq 3$, $ (n-2)^k R_{n+1,n-1}(k)$ equals the number of connected $sl_n$ webs that lack an internal cycle and have a boundary string with $n(k+1)$ total 1's.
\end{conjecture}

The specific interpretation of $(n+1,n-1)$-coral diagrams that motivates Conjecture \ref{thm: sln webs, constant} is shown in Figure \ref{fig: sl4 webs} for the relatively easy case of $n=4$.  On the left we show the base for our $(5,3)$-coral diagrams, here interpreted as part of an $sl_4$ web, and on the right we show two non-equivalent pieces of $sl_4$ webs that both function as valid choices for each $5$-star.  In the $sl_4$ case, the reason that there aren't additional non-equivalent variations of these pieces follows from Kim's relations for $sl_4$ webs \cite{Kim}.  The relevant member of these relations is shown in Figure \ref{fig: sl4 web relation}; notice that this relation also sees a direct analogue in the more general $sl_n$ web relations of Cautis, Kamnitzer, and Morrison (Relation 2.6 of \cite{CKM}).  In the general $sl_n$ case, we conjecture the existence of one valid base and $n-2$ valid choices for each star.  It is these $n-2$ non-equivalent choices for each of the $k$ total $(n+1)$-stars in our coral diagrams that leads to the unusual ``correction factor" of $(n-2)^k$ in Conjecture \ref{thm: sln webs, constant}.

\begin{figure}[h!]
\begin{tikzpicture}
	[scale=.8,auto=left,every node/.style={circle,fill=black,inner sep=1.7pt}]
	\node (n1) at (0,0) {};
	\node (n2) at (-1.5,3.5) {};
	\node (n3) at (0,1) {};
	\node (n4) at (1.5,2) {};
	\node (n5) at (-.5,2) {};
	\node (n6) at (.5,2) {};
	\draw[->-, thick, bend left=25] (n1) to (n2);
	\draw[double, thick] (n1) to (n3);
	\draw[->-, thick, bend right=15] (n1) to (n4);
	\draw[->-, thick] (n3) to (n5);
	\draw[->-, thick] (n3) to (n6);
\end{tikzpicture}
\hspace{1in}
\begin{tikzpicture}
	[scale=.8,auto=left,every node/.style={circle,fill=black,inner sep=1.7pt}]
	\node (n1) at (0,0) {};
	\node (n2) at (-1,1) {};
	\node (n3) at (1,1) {};
	\node (n4) at (-1.5,2) {};
	\node (n5) at (-0.5,2) {};
	\node (n6) at (0.5,2) {};
	\node (n7) at (1.5,2) {};
	\node (n8) at (0,3) {};
	\node (n9) at (1,3) {};
	\draw[thick,double] (n1) to (n2);
	\draw[->-, thick] (n3) to (n1);
	\draw[->-, thick] (n2) to (n4);
	\draw[->-, thick] (n2) to (n5);
	\draw[thick,double] (n3) to (n6);
	\draw[->-, thick] (n3) to (n7);
	\draw[->-, thick] (n6) to (n8);
	\draw[->-, thick] (n6) to (n9);
\end{tikzpicture}
\hspace{.15in}
\raisebox{36pt}{\scalebox{1.5}{$\ncong$}}
\hspace{.05in}
\begin{tikzpicture}
	[scale=.8,auto=left,every node/.style={circle,fill=black,inner sep=1.7pt}]
	\node (n1) at (0,0) {};
	\node (n2) at (-1,1) {};
	\node (n3) at (1,1) {};
	\node (n4) at (-1.5,2) {};
	\node (n5) at (-0.5,2) {};
	\node (n6) at (0.5,2) {};
	\node (n7) at (1.5,2) {};
	\node (n8) at (-2,3) {};
	\node (n9) at (-1,3) {};
	\draw[->-, thick] (n2) to (n1);
	\draw[double, thick] (n3) to (n1);
	\draw[double, thick] (n2) to (n4);
	\draw[->-, thick] (n2) to (n5);
	\draw[->-, thick] (n3) to (n6);
	\draw[->-, thick] (n3) to (n7);
	\draw[->-, thick] (n4) to (n8);
	\draw[->-, thick] (n4) to (n9);
\end{tikzpicture}
\caption{$sl_4$ web interpretation for $R_{5,3}$, with base (left) and two non-equivalent choices for each p-star (right)}
\label{fig: sl4 webs}
\end{figure}
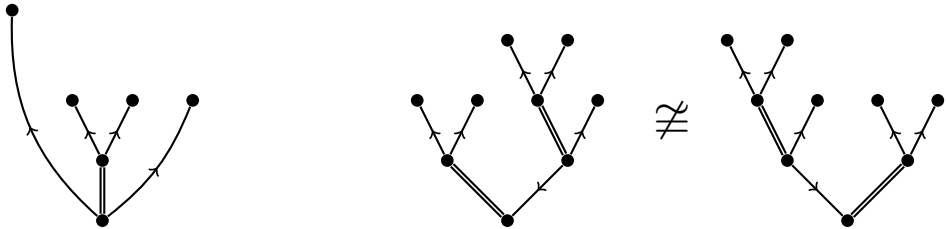

\begin{figure}[h!]
\begin{tikzpicture}
	[scale=.8,auto=left,every node/.style={circle,fill=black,inner sep=1.7pt}]
	\node (n0) at (0,-.8) {};	
	\node (n1) at (0,0) {};
	\node (n2) at (-.5,.8) {};
	\node (n3) at (.5,.8) {};
	\node (n4) at (-1,1.6) {};
	\node (n5) at (0,1.6) {};
	\draw[->-, thick] (n1) to (n0);
	\draw[double, thick] (n1) to (n2);
	\draw[->-, thick] (n1) to (n3);
	\draw[->-, thick] (n2) to (n4);
	\draw[->-, thick] (n2) to (n5);
\end{tikzpicture}
\hspace{.15in}
\raisebox{30pt}{\scalebox{1.5}{$\cong$}}
\hspace{.15in}
\begin{tikzpicture}
	[scale=.8,auto=left,every node/.style={circle,fill=black,inner sep=1.7pt}]
	\node (n0) at (0,-.8) {};
	\node (n1) at (0,0) {};
	\node (n2) at (-.5,.8) {};
	\node (n3) at (.5,.8) {};
	\node (n4) at (0,1.6) {};
	\node (n5) at (1,1.6) {};
	\draw[->-, thick] (n1) to (n0);
	\draw[->-, thick] (n1) to (n2);
	\draw[double, thick] (n1) to (n3);
	\draw[->-, thick] (n3) to (n4);
	\draw[->-, thick] (n3) to (n5);
\end{tikzpicture}
\caption{One of Kim's $sl_4$ web relations}
\label{fig: sl4 web relation}
\end{figure}
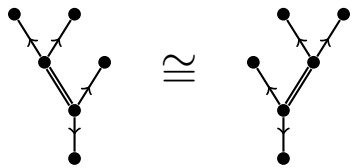

The reason that Conjecture \ref{thm: sln webs, constant} cannot be proven at this point is that we lack a general proof of the fact that, in the $(n+1,n-1)$-coral diagram interpretation, there is precisely one valid choice for our base and precisely $(n-2)$ non-equivalent choices for each $(n+1)$-star.  Furthermore, it would need to be shown that the different ``pieces" of the coral diagram do not interact and allow for additional relations  that cannot be localized to differences in the base or within a single coral.

In light of the generating $sl_n$ web relations of \cite{CKM}, it appears that the only relation capable of directly relating a pair of distinct tree webs is the aforementioned $I = H$ relation (Relation 2.6, \cite{CKM}).  This relation should allow for a direct justification of the fact that there are $n-2$ distinct choices for each $(n+1)$-star in the above interpretation.  However, it would still need to be shown that two tree webs cannot be connected via a string of other relations that pass through at least one non-tree $sl_n$ web.


On a more basic level, notice that the local relations of \cite{CKM} do not result in simple global conditions for determining whether an $sl_n$ web is a member of a given basis: in the case of $sl_3$ webs the non-elliptic condition provides an easy way to recognize whether a given web is an element of the non-elliptic basis, but there is no similarly tractable condition for recognizing basis webs in the $n>3$ case.  Also notice that the constructive $sl_n$ web bases developed in \cite{West} and \cite{Fon} aren't well-suited to proving our conjecture, as there isn't a reasonable way to determine which inputs to their growth algorithms yield webs that are connected trees.  Even if these bases could be used to show that there are at least $n+2$ non-equivalent choices for each $(n+1)$-star, they cannot be easily applied to prove there are not additional tree webs that are equivalent to elements of the resulting set.

We also conjecture that Proposition \ref{thm: sl3 webs, non-constant string} for $sl_3$ webs may be generalized to $sl_n$ webs with a boundary string of the form $(j \ 1 \ 1 \hdots 1)$.  The desired result is given in Conjecture \ref{thm: sln webs, non-constant}.  If true, this result would give a combinatorial interpretation of $R_{p,r}(k)$ for all $k \geq 0$ whenever $1 \leq r \leq p-2$.  

\begin{conjecture}
\label{thm: sln webs, non-constant}
For any $n \geq 3$ and any $1 \leq j \leq n-1$, $(n-2)^k R_{n-1,n-j}(k)$ equals the number of connected $sl_n$ webs that lack an internal cycle and have a boundary string consisting of one $j$ followed by $nk + n - j$ consecutive 1's.
\end{conjecture}


\bibliographystyle{amsplain}
\bibliography{biblio}

\end{document}